\documentclass[12pt]{amsart}
\usepackage{amsmath,amssymb,amsthm,amscd,amsfonts,geometry,color}
\renewcommand{\labelenumi}{\rm(\theenumi)}

\theoremstyle{plain}
\newtheorem{thm}{Theorem}[section]
\newtheorem{prop}[thm]{Proposition}
\newtheorem{cor}[thm]{Corollary}
\newtheorem{lem}[thm]{Lemma}

\theoremstyle{remark}
\newtheorem{remark}{Remark}

\theoremstyle{definition}

\newcommand{\R}{\mathbb{R}}                       % real numbers
\newcommand{\N}{\mathbb{N}}                       % natural numbers
\newcommand{\lp}{\operatorname{L^p}}            % L^p space

\begin{document}

\title[Compactness of averaging operators on Banach function spaces]{Compactness of averaging operators on Banach function spaces}
\author{Katsuhisa Koshino}
\address[Katsuhisa Koshino]{Faculty of Engineering, Kanagawa University, Yokohama, 221-8686, Japan}
\email{ft160229no@kanagawa-u.ac.jp}
\maketitle

\begin{abstract}
Let $X$ be a Borel metric measure space such that each closed ball is of positive and finite measure.
In this paper, we give a sufficient and necessary condition for averaging operators on a Banach function space $E(X)$ on $X$ to be compact.
As a corollary, we show that the averaging operators on the Lorentz space $L^{p,q}(X)$ of $X$ is compact if and only if $X$ is bounded,
 in the case where $X$ is a doubling and Borel-regular metric measure space with some continuity between metric and measure.
\end{abstract}

\section{Introduction}

Throughout this paper, let $X = (X,d,\mu)$ be a Borel metric measure space with a metric $d$ and a measure $\mu$ such that for any point $x \in X$ and any positive number $r > 0$, $0 < \mu(B(x,r)) < \infty$,
 where $B(x,r)$ is the closed ball centered at $x$ with radius $r$.
The symbols $\R$ and $\N$ are the set of real numbers and the one of positive integers, respectively.
Denote the space of measurable functions on $X$ by $L^0(X)$,
 and the characteristic function on a subset $A \subset X$ by $\chi_A$.
Let $E(X) = (E(X),\|\cdot\|_E) \subset L^0(X)$ be a normed linear space with a norm $\|\cdot\|_E$ satisfying the following conditions:
\begin{itemize}
 \item[(B1)] for each $f \in E(X)$, $\|f\|_E = \||f|\|_E$,
 and $\|f\|_E = 0$ if and only if $f = 0$ a.e.;
 \item[(B2)] for any $f, g \in E(X)$, if $0 \leq g \leq f$ a.e.,
 then $\|g\|_E \leq \|f\|_E$;
 \item[(B3)] for all $f, f_n \in E(X)$, $n \in \N$, if $0 \leq f_n \uparrow f$ a.e.,
 then $\|f_n\|_E \uparrow \|f\|_E$;
 \item[(B4)] for every measurable set $A \subset X$, if $\mu(A) < \infty$,
 then $\chi_A \in E(X)$ and $\|\chi_A\|_E < \infty$;
 \item[(B5)] for each measurable set $A \subset X$ with $\mu(A) < \infty$, there exists $\alpha(A) \in (0,\infty)$ such that $\int_A |f| d\mu \leq \alpha(A)\|f\|_E$ for any $f \in E(X)$.
\end{itemize}
Here we say that $E(X)$ is a \textit{Banach function space} on $X$.
It is known that $E(X)$ is a Banach space, see Theorem~3.1.3 of \cite{EE}.
From now on, fix $r > 0$.
Define $A_r : E(X) \to L^0(X)$ by for any $f \in E(X)$,
 $$A_rf(x) = \frac{1}{\mu(B(x,r))}\int_{B(x,r)} f(y) d\mu(y),$$
 which is called the \textit{averaging operator} on $E(X)$.\footnote{According to the condition (B5), every $f \in E(X)$ is integrable on a measurable subset of finite measure,
 and hence $A_rf \in L^0(X)$, refer to \cite[Theorem~8.3]{Ye}.}
Averaging operators play important roles in analysis.
For instance, using averaging operators, P.~G\'{o}rka and A.~Macios \cite{GM}, and the author \cite{Kos21} characterized compact sets in Lebesgue spaces on metric measure spaces,
 which are generalizations of the Kolmogorov-Riesz theorem \cite{Kol,R}.
The boundedness of averaging operators on Lebesgue spaces of metric measure spaces are studied in the paper \cite{Al}.
In this paper, we shall give a sufficient and necessary condition for the averaging operator to be compact in the case that $A_r(E(X)) \subset E(X)$.

A metric measure space $X$ has \textit{the Vitali covering property} provided that the following is satisfied:
\begin{itemize}
 \item Let $A$ be any subset of $X$ and $\mathcal{B}$ be any family of closed balls centered in $A$ with uniformly bounded radii such that for every $x \in A$, there is a ball belonging to $\mathcal{B}$ centered at $x$ and
 $$\inf\{s > 0 \mid B(x,s) \in \mathcal{B}\} = 0.$$
 Then there is a subfamily $\mathcal{B}' \subset \mathcal{B}$ consisting of pairwise disjoint countable closed balls such that $\mu(A \setminus \bigcup \mathcal{B}') = 0$.
\end{itemize}
It is said that a norm $\|\cdot\|_E$ on $E(X)$ is \textit{absolutely continuous} if $\|f\chi_{A_i}\|_E \to 0$ as $i \to \infty$ for each $f \in E(X)$ whenever $\{A_i\}_{i \in \N}$ is a sequence of measurable sets in $X$ such that $\chi_{A_i} \to 0$ a.e..
In the case where $\mu(X) < \infty$, $\|\cdot\|_E$ is called to be \textit{absolutely continuous with respect to $1$} provided that the above convergence holds for the constant $1$.
As is easily observed,
 if $\|\cdot\|_E$ is absolutely continuous with respect to $1$,
 then for any $\epsilon > 0$, there is $\delta > 0$ such that for any measurable subset $A \subset X$ with $\mu(A) < \delta$, $\|\chi_A\|_E < \epsilon$, see \cite[Lemma~3.1.13]{EE}.
We will give a sufficient condition for the compactness of averaging operators as follows:

\begin{thm}\label{suff.}
Suppose that $A_r(E(X)) \subset E(X)$ and the following conditions hold:
\begin{enumerate}
 \item $X$ satisfies the Vitali covering property;
 \item $\inf\{\mu(B(x,r)) \mid x \in X\} > 0$;
 \item for each $x \in X$, $\mu(B(y,r)) \to \mu(B(x,r))$ and $\alpha(B(x,r) \triangle B(y,r)) \to 0$ as $y \to x$;\footnote{For subsets $A, B \subset X$, put $A \triangle B = (A \setminus B) \cup (B \setminus A)$.}
 \item the norm $\|\cdot\|_E$ is absolutely continuous with respect to $1$.
\end{enumerate}
If $\mu(X) < \infty$,
 then $A_r : E(X) \to E(X)$ is compact.
\end{thm}

For a positive number $s > 0$, a metric measure space $X$ is \textit{$s$-doubling} if there is $\gamma \geq 1$ such that $\mu(B(x,2s)) \leq \gamma\mu(B(x,s))$ for every point $x \in X$.
Recall that $X$ is $s$-doubling if and only if $\inf\{\mu(B(x,s))/\mu(B(x,2s)) \mid x \in X\} > 0$.

\begin{remark}
We say that a metric measure space $X$ is \textit{doubling} if there exists $\gamma \in [1,\infty)$ such that $\mu(B(x,2s)) \leq \gamma\mu(B(x,s))$ for all points $x \in X$ and all positive number $s \in (0,\infty)$.
Clearly, $X$ is $s$-doubling for any $s > 0$ if it is doubling.
It is well known that $X$ has the Vitali covering porperty when it is doubling and \textit{Borel-regular},
 which means that each subset of $X$ is contained in some Borel set with the same measure, see the Vitali covering theorem \cite[Theorem~6.20]{Ye}.
\end{remark}

The boundedness of $X$ is necessary for $A_r$ to be compact.

\begin{thm}\label{nec.}
Suppose that $A_r(E(X)) \subset E(X)$ and the following conditions are satisfied:
\begin{enumerate}
 \item $X$ is $r$-doubling;
 \item $\sup\Big\{\Big\|\frac{\alpha(B(x,r))\chi_{B(x,2r)}}{\mu(B(x,r))}\Big\|_E \ \Big| \ x \in X\Big\} < \infty$.
\end{enumerate}
If $A_r$ is compact,
 then $X$ is bounded.
\end{thm}

Given any $f \in L^0(X)$, we define a function $\mu_f : [0,\infty) \to [0,\infty]$ by
 $$\mu_f(t) = \mu(\{x \in X \mid |f(x)| > t\}),$$
 which is called the \textit{distribution function} of $f$.
The \textit{rearrangement} of $f$ is a function $f^\ast : [0,\infty) \to [0,\infty]$ defined by
 $$f^\ast(t) = \inf\{s \in [0,\infty) \mid \mu_f(s) \leq t\}.$$
For $p, q \in [1,\infty]$, let $L^{p,q}(X)$ be the space consisting of functions $f \in L^0(X)$ satisfying
 $$\bigg(\int_0^\infty t^{q/p - 1}f^\ast(t)^q dt\bigg)^{1/q} < \infty \text{ if } q < \infty, \text{ or } \sup_{t \in (0,\infty)} t^{1/p}f^\ast(t) < \infty \text{ if } q = \infty,$$
 which we call the \textit{Lorentz space} on $X$.
This was introduce in \cite{L1,L2}.
For each $f \in L^{p,q}(X)$, put
 $$\|f\|_{p,q} = \left\{
 \begin{array}{ll}
  \bigg(\int_0^\infty t^{q/p - 1}f^\ast(t)^q dt\bigg)^{1/q} &\text{if } q < \infty,\\
  \sup_{t \in (0,\infty)} t^{1/p}f^\ast(t) &\text{if } q = \infty.
 \end{array}
 \right.$$
Note that $L^{p,p}(X) = (L^{p,p}(X),\|\cdot\|_{p,p})$ is coincident with the Lebesgue space $L^p(X) = (L^p(X),\|\cdot\|_p)$,
 which is a Banach function space, according to the definition and \cite[Proposition~3.2.5]{EE}.
It is known that $\|\cdot\|_{p,q}$ is a norm on $L^{p,q}(X)$ and $L^{p,q}(X) = (L^{p,q}(X),\|\cdot\|_{p,q})$ is a Banach function space in the case that $1 \leq q \leq p$, refer to Theorem~4.3 of \cite{BS}.
Moreover, let
 $$f^{\ast\ast}(t) = \frac{1}{t}\int_0^t f^\ast(s) ds$$
 for each $t > 0$,
 and let
 $$\|f\|_{(p,q)} = \left\{
 \begin{array}{ll}
  \bigg(\int_0^\infty t^{q/p - 1}f^{\ast\ast}(t)^q dt\bigg)^{1/q} &\text{if } q < \infty,\\
  \sup_{t \in (0,\infty)} t^{1/p}f^{\ast\ast}(t) &\text{if } q = \infty.
 \end{array}
 \right.$$
 so $\|\cdot\|_{(p,q)}$ is an equivalent norm to $\|\cdot\|_{p,q}$ and $L^{p,q}(X) = (L^{p,q}(X),\|\cdot\|_{(p,q)})$ is also a Banach function space if $p > 1$ and $q \geq 1$, see \cite[Lemma~3.4.6 and Theorem~3.4.7]{EE}.
From now on, the space $L^{p,q}(X) = (L^{p,q}(X),\|\cdot\|_E)$ is equipped with either of the above norms $\|\cdot\|_{p,q}$ or $\|\cdot\|_{(p,q)}$.\footnote{When $q < p = \infty$,
 $L^{p,q}(X) = \{0\}$.
 There are not always norms equivalent to $\|\cdot\|_{p,q}$ if $1 = p < q \leq \infty$.}
As a corollary of Theorems~\ref{suff.} and \ref{nec.}, we can establish the following:

\begin{cor}\label{equiv.lorentz}
Let $X$ be an $s$-doubling metric measure space having the Vitali covering property for all $s \geq r$ such that for each point $x \in X$, $\mu(B(x,r) \triangle B(y,r)) \to 0$ as $y \to x$,
 and let $p \in (1,\infty)$ and $q \in [1,\infty]$.
Then $A_r : L^{p,q}(X) \to L^{p,q}(X)$ is compact if and only if $X$ is bounded.
\end{cor}

\section{Proof of Theorems~\ref{suff.} and \ref{nec.}}

In this section, we shall prove Theorems~\ref{suff.} and \ref{nec.}.

\begin{lem}\label{equiconti.}
Let $x \in X$.
Suppose that $\mu(B(y,r)) \to \mu(B(x,r))$ and $\alpha(B(x,r) \triangle B(y,r)) \to 0$ as $y \to x$.
Then for each bounded subset $F \subset E(X)$, $\{A_rf \mid f \in F\}$ is equicontinuous at $x$.
\end{lem}

\begin{proof}
To show the equicontinuity of $\{A_rf \mid f \in F\}$, fix any $\epsilon > 0$.
We need only to prove that there exists $\delta > 0$ such that if $d(x,y) < \delta$,
 then $|A_rf(x) - A_rf(y)| < \epsilon$ for any $f \in F$.
Let $c_1 = \sup\{\|f\|_E \mid f \in F\}$, $c_2 = \mu(B(x,r))$ and $c_3 = \alpha(B(x,r))$.
By the assmption, we can obtain a positive number $\delta > 0$ so that if $d(x,y) < \delta$,
 then $|1/\mu(B(x,r)) - 1/\mu(B(y,r))| < \epsilon/(2c_1c_3)$ and $\alpha(B(x,r) \triangle B(y,r)) < c_1c_2c_3\epsilon/(c_1(2c_1c_3 + c_2\epsilon))$.
Observe that for any $y \in X$ with $d(x,y) < \delta$,
\begin{align*}
 |A_rf(x) - A_rf(y)| &= \bigg|\frac{1}{\mu(B(x,r))}\int_{B(x,r)} f(z) d\mu(z) - \frac{1}{\mu(B(y,r))}\int_{B(y,r)} f(z) d\mu(z)\bigg|\\
 &\leq \bigg|\frac{1}{\mu(B(x,r))} - \frac{1}{\mu(B(y,r))}\bigg|\bigg|\int_{B(x,r)} f(z) d\mu(z)\bigg|\\
 &\ \ \ \ \ \ \ \ + \frac{1}{\mu(B(y,r))}\bigg|\int_{B(x,r)} f(z) d\mu(z) - \int_{B(y,r)} f(z) d\mu(z)\bigg|\\
 &\leq \bigg|\frac{1}{\mu(B(x,r))} - \frac{1}{\mu(B(y,r))}\bigg|\bigg(\int_{B(x,r)} |f(z)| d\mu(z)\bigg)\\
 &\ \ \ \ \ \ \ \ + \frac{1}{\mu(B(y,r))}\int_{B(x,r) \triangle B(y,r)} |f(z)| d\mu(z)\\
 &\leq \bigg|\frac{1}{\mu(B(x,r))} - \frac{1}{\mu(B(y,r))}\bigg|\alpha(B(x,r))\|f\|_E\\
 &\ \ \ \ \ \ \ \ + \frac{1}{\mu(B(y,r))}\alpha(B(x,r) \triangle B(y,r))\|f\|_E\\
 &< \frac{\epsilon}{2c_1c_3} \cdot c_3 \cdot c_1 + \bigg(\frac{\epsilon}{2c_1c_3} + \frac{1}{c_2}\bigg) \cdot \frac{c_1c_2c_3\epsilon}{c_1(2c_1c_3 + c_2\epsilon)} \cdot c_1 = \epsilon.
\end{align*}
The proof is completed.
\end{proof}

\begin{lem}\label{bdd.}
Let $x \in X$.
If $F$ is a bounded subset of $E(X)$,
 then the set $\{A_rf(x) \mid f \in F\}$ is bounded.
\end{lem}

\begin{proof}
Letting $c = \sup\{\|f\|_E \mid f \in F\}$, we have that for every $f \in F$,
\begin{align*}
 |A_rf(x)| &= \bigg|\frac{1}{\mu(B(x,r))}\int_{B(x,r)} f(y) d\mu(y)\bigg| \leq \frac{1}{\mu(B(x,r))}\int_{B(x,r)} |f(y)| d\mu(y)\\
 &\leq \frac{\alpha(B(x,r))\|f\|_E}{\mu(B(x,r))} \leq \frac{c\alpha(B(x,r))}{\mu(B(x,r))},
\end{align*}
 which implies that $\{A_rf(x) \mid f \in F\}$ is bounded.
\end{proof}

Using the above Lemmas~\ref{equiconti.} and \ref{bdd.}, we will prove Theorem~\ref{suff.}.

\begin{proof}[Proof of Theorem~\ref{suff.}]
Let $F$ be a bounded subset of $E(X)$.
To investigate the total boundedness of $\{A_rf \mid f \in F\}$ in $E(X)$, fix any $\epsilon > 0$.
Remark that $\|\chi_X\|_E < \infty$ and $\alpha(X) < \infty$.
According to Lemma~\ref{equiconti.}, for each point $x \in X$, there exists $\delta(x) > 0$ such that if $d(x,y) < \delta(x)$,
 then $|A_rf(x) - A_rf(y)| < \epsilon/(4\|\chi_X\|_E)$ for any $f \in F$.
Put $c_1 = \sup\{\|f\|_E \mid f \in F\}$ and $c_2 = \inf\{\mu(B(x,r)) \mid x \in X\}$.
Combining the Vitali covering property of $X$, the finiteness of the measure of $X$ and the absolute continuity of $\|\cdot\|_E$, we can choose pairwise disjoint finitely many closed balls $B(x_i,r_i)$ centered at $x_i \in X$ with radius $r_i < \delta(x_i)$, $1 \leq i \leq n$, so that $\|\chi_{X \setminus \bigcup_{i = 1}^n B(x_i,r_i)}\|_E < c_2\epsilon/(2c_1\alpha(X))$.
For every $i \in \{1, \cdots, n\}$, the set $\{A_rf(x_i) \mid f \in F\}$ is bounded due to Lemma~\ref{bdd.},
 so we can find a finite subset $F_i \subset \R$ such that for each $f \in F$, there is a number $a_i(f) \in F_i$ such that $|A_rf(x_i) - a_i(f)| \leq \epsilon/(4\|\chi_X\|_E)$.
Note that
 $$\Bigg\{\sum_{i = 1}^n a_i\chi_{B(x_i,r_i)} \ \Bigg| \ \text{for each } i \in \{1, \cdots, n\}, a_i \in F_i\Bigg\}$$
 is a set consisting of finitely many simple functions of $E(X)$.

Take any $f \in F$.
For every $i \in \{1, \cdots, n\}$ and every $x \in B(x_i,r_i)$, since $r_i < \delta(x_i)$,
 $$|A_rf(x) - a_i(f)| \leq |A_rf(x) - A_rf(x_i)| + |A_rf(x_i) - a_i(f)| < \frac{\epsilon}{2\|\chi_X\|_E}.$$
Hence we get that
\begin{align*}
 \Bigg\|\sum_{i = 1}^n (A_rf)\chi_{B(x_i,r_i)} - \sum_{i = 1}^n a_i(f)\chi_{B(x_i,r_i)}\Bigg\|_E &= \Bigg\|\sum_{i = 1}^n (A_rf - a_i(f))\chi_{B(x_i,r_i)}\Bigg\|_E\\
 &= \Bigg\|\sum_{i = 1}^n |A_rf - a_i(f)|\chi_{B(x_i,r_i)}\Bigg\|_E\\
 &\leq \frac{\epsilon}{2\|\chi_X\|_E}\|\chi_{\cup_{i = 1}^n B(x_i,r_i)}\|_E \leq \frac{\epsilon}{2\|\chi_X\|_E}\|\chi_X\|_E = \frac{\epsilon}{2}.
\end{align*}
On the other hand, for each point $x \in X \setminus \bigcup_{i = 1}^n B(x_i,r_i)$,
$$|A_rf(x)| = \bigg|\frac{1}{\mu(B(x,r))}\int_{B(x,r)} f(y) d\mu(y)\bigg| \leq \frac{1}{c_2}\int_X |f(y)| d\mu(y) \leq \frac{\alpha(X)\|f\|_E}{c_2} \leq \frac{c_1\alpha(X)}{c_2},$$
 and then
\begin{align*}
 \Bigg\|A_rf - \sum_{i = 1}^n (A_rf)\chi_{B(x_i,r_i)}\Bigg\|_E &= \|(A_rf)\chi_{X \setminus \bigcup_{i = 1}^n B(x_i,r_i)}\|_E \leq \frac{c_1\alpha(X)}{c_2}\|\chi_{X \setminus \bigcup_{i = 1}^n B(x_i,r_i)}\|_E\\
 &\leq \frac{c_1\alpha(X)}{c_2} \cdot \frac{c_2\epsilon}{2c_1\alpha(X)} = \frac{\epsilon}{2}.
\end{align*}
It follows that
\begin{align*}
 &\Bigg\|A_rf - \sum_{i = 1}^n a_i(f)\chi_{B(x_i,r_i)}\Bigg\|_E\\
 &\ \ \ \ \leq \Bigg\|A_rf - \sum_{i = 1}^n (A_rf)\chi_{B(x_i,r_i)}\Bigg\|_E + \Bigg\|\sum_{i = 1}^n (A_rf)\chi_{B(x_i,r_i)} - \sum_{i = 1}^n a_i(f)\chi_{B(x_i,r_i)}\Bigg\|_E < \epsilon.
\end{align*}
Consequently, $\{A_rf \mid f \in F\}$ is totally bounded,
 which means that $A_r$ is compact.
We complete the proof.
\end{proof}

Next, we show Theorem~\ref{nec.}.

\begin{proof}[Proof of Theorem~\ref{nec.}]
Let $c = \inf\{\mu(B(x,r))/\mu(B(x,2r)) \mid x \in X\}$.
Supposing that $X$ is not bounded, we can choose a sequence $\{x_n\} \subset X$ so that for all points $m, n \in \N$ with $m \neq n$, $d(x_m,x_n) > 4r$.
Define
 $$f_n = \frac{\alpha(B(x_n,r))\chi_{B(x_n,2r)}}{\mu(B(x_n,r))}.$$
According to (2) of the assumption, $\{f_n \mid n \in \N\}$ is a bounded subset of $E(X)$.
Observe that for each $n \in \N$ and each $x \in B(x_n,r)$,
\begin{align*}
 A_rf_n(x) &= \frac{1}{\mu(B(x,r))}\int_{B(x,r)} f_n(y) d\mu(y) = \frac{1}{\mu(B(x,r))}\int_{B(x,r)} \frac{\alpha(B(x_n,r))\chi_{B(x_n,2r)}(y)}{\mu(B(x_n,r))} d\mu(y)\\
 &\geq \frac{1}{\mu(B(x,2r))}\int_{B(x,r)} \frac{\alpha(B(x_n,r))\chi_{B(x_n,2r)}(y)}{\mu(B(x_n,r))} d\mu(y) = \frac{\mu(B(x,r))\alpha(B(x_n,r))}{\mu(B(x,2r))\mu(B(x_n,r))}\\
 &\geq \frac{c\alpha(B(x_n,r))}{\mu(B(x_n,r))},
\end{align*}
 and for each $x \in X \setminus B(x_n,3r)$,
 $$A_rf_n(x) = \frac{1}{\mu(B(x,r))}\int_{B(x,r)} f_n(y) d\mu(y) = \frac{1}{\mu(B(x,r))}\int_{B(x,r)} \frac{\alpha(B(x_n,r))\chi_{B(x_n,2r)}(y)}{\mu(B(x_n,r))} d\mu(y) = 0.$$
Hence for any $m, n \in \N$, if $m \neq n$,
 then
 \begin{align*}
  \|A_rf_n - A_rf_m\|_E &\geq \frac{1}{\alpha(B(x_n,r))}\int_{B(x_n,r)} |A_rf_n(y) - A_rf_m(y)| d\mu(y)\\
  &\geq \frac{1}{\alpha(B(x_n,r))}\int_{B(x_n,r)} \frac{c\alpha(B(x_n,r))}{\mu(B(x_n,r))} d\mu(y) = c.
 \end{align*}
Therefore $A_r(\{f_n \mid n \in \N\})$ is not relatively compact,
 which contradicts to the compactness of $A_r$.
As a consequent, $X$ is bounded.
\end{proof}

\section{Proof of Corollary~\ref{equiv.lorentz}}

This section is devoted to proving Corollary~\ref{equiv.lorentz}.
We have the following:

\begin{lem}\label{doubling}
The following are equivalent:
\begin{enumerate}
\renewcommand{\labelenumi}{(\roman{enumi})}
 \item $\mu(X) < \infty$ and $\inf\{B(x,r) \mid x \in X\} > 0$;
 \item $\mu(X) < \infty$ and $\inf\{B(x,s) \mid x \in X\} > 0$ for any $s \geq r$;
 \item $X$ is bounded and $s$-doubling for each $s \geq r$.
\end{enumerate}
\end{lem}

\begin{proof}
The implication (i) $\Rightarrow$ (ii) is clear.
Next, we prove (ii) $\Rightarrow$ (iii).
Assuming that $X$ is unbounded,
 so we can choose a sequence $\{x_n\} \subset X$ so that $d(x_i,x_j) > 2r$ for any $i \neq j$.
Since $\inf\{\mu(B(x,r)) \mid x \in X\} > 0$,
 $$\infty > \mu(X) \geq \mu\Bigg(\bigcup_{n \in \N} B(x_n,r)\Bigg) = \sum_{n \in \N} \mu(B(x_n,r)) \geq \aleph_0 \cdot \inf\{\mu(B(x,r)) \mid x \in X\} = \infty,$$
 which is a contradiction.
Thus $X$ is bounded.
Moreover, letting any $s \geq r$, put $\epsilon = \inf\{\mu(B(x,s)) \mid x \in X\}$.
Take $\gamma \geq 1$ such that $\mu(X) \leq \gamma\epsilon$,
 so we have that for every $z \in X$,
 $$\mu(B(z,2s)) \leq \mu(X) \leq \gamma\epsilon = \gamma\inf\{\mu(B(x,s)) \mid x \in X\} \leq \gamma\mu(B(z,s)).$$
Therefore $X$ is $s$-doubling.

Finally, the implication (iii) $\Rightarrow$ (i) will be shown by the same argument as \cite[Lemma~2.3]{Kos21}.
Indeed, since $X$ is $s$-doubling for every $s \geq r$,
 we can choose $\gamma_n \geq 1$ for each $n \in \N$ so that $\mu(B(x,2^nr)) \leq \gamma_n\mu(B(x,2^{n - 1}r))$ for all $x \in X$.
It follows from the boundedness of $X$ that there is $m \in \N$ such that $X = B(x,2^mr)$ for every point $x \in X$ and that $\mu(X) < \infty$.
Observe that
 $$\mu(B(x,r)) \geq (\gamma_1\gamma_2 \cdots \gamma_m)^{-1}\mu(B(x,2^mr)) \geq (\gamma_1\gamma_2 \cdots \gamma_m)^{-1}\mu(X) > 0,$$
 which implies (i).
\end{proof}

Combining \cite[Proposition~2.2]{GS} with the above lemma, we can establish the following proposition:

\begin{prop}\label{tot.bdd.}
The following are equivalent:
\begin{enumerate}
\renewcommand{\labelenumi}{(\roman{enumi})}
 \item $X$ is totally bounded;
 \item $\mu(X) < \infty$ and $\inf\{B(x,s) \mid x \in X\} > 0$ for every $s > 0$;
 \item $X$ is bounded and $s$-doubling for any $s > 0$.
\end{enumerate}
\end{prop}

To show the boundedness of the averaging operators on $L^{p,q}(X)$, we will use the following lemma:

\begin{lem}\label{distribution}
Let $X$ be an $s$-doubling metric measure space with the Vitali covering property for any $s \geq r$.
Then there exists $c \geq 2$ such that for each $f \in L^0(X)$ which is integrable on all measurable set of finite measure and for each $t > 0$,
 $$\mu_{A_rf}(ct) \leq \frac{1}{t}\int_{\{x \in X \mid |f(x)| > t\}} |f(x)| d\mu(x).$$
\end{lem}

\begin{proof}
Since $X$ is $s$-doubling for all $s \geq r$,
 there is $\gamma_n \geq 1$ for each positive integer $1 \leq n \leq 3$ such that $\mu(B(x,2^nr)) \leq \gamma_n\mu(B(x,2^{n - 1}r))$ for all $x \in X$.
Put $c = \gamma_1\gamma_2\gamma_3 + 1$ and $Y = \{x \in X \mid |A_rf(x)| > ct\}$.
Then for every $x \in Y$,
 $$ct\mu(B(x,r)) < \int_{B(x,r)} |f(y)| d\mu(y) \leq \int_{B(x,r) \cap \{y \in X \mid |f(y)| > t\}} |f(y)| d\mu(y) + t\mu(B(x,r)).$$
Since $X$ has the Vitali covering property,
 we can find a pairwise disjoint countable family $\{B(x_i,r_i) \mid i \in I\}$ of closed balls,
 where $I \subset \N$, such that for each $i \in I$, $x_i \in Y$ and $r_i \leq r$,
 and $\mu(Y \setminus \bigcup_{i \in I} B(x_i,r_i)) = 0$.
By Theorem~6.6 of \cite{Ye}, there is a pairwise disjoint subfamily $\{B(x_j,r) \mid j \in J\} \subset \{B(x_i,r) \mid i \in I\}$,
 where $J \subset I$, such that $\bigcup_{i \in I} B(x_i,r) \subset \bigcup_{j \in J} B(x_j,5r)$.
Observe that
\begin{align*}
 \mu_{A_rf}(ct) &= \mu(Y) = \sum_{i \in I} \mu(Y \cap B(x_i,r_i)) \leq \sum_{i \in I} \mu(B(x_i,r_i)) = \mu\Bigg(\bigcup_{i \in I} B(x_i,r_i)\Bigg)\\
 &\leq \mu\Bigg(\bigcup_{i \in I} B(x_i,r)\Bigg) \leq \mu\Bigg(\bigcup_{j \in J} B(x_j,5r)\Bigg) \leq \sum_{j \in J} \mu(B(x_j,5r))\\
 &\leq \gamma_1\gamma_2\gamma_3\sum_{j \in J} \mu(B(x_j,r)) < \gamma_1\gamma_2\gamma_3\sum_{j \in J} \frac{1}{(c - 1)t}\int_{B(x_j,r) \cap \{y \in X \mid |f(y)| > t\}} |f(y)| d\mu(y)\\
 &\leq \frac{1}{t}\int_{\{x \in X \mid |f(x)| > t\}} |f(x)| d\mu(x).
\end{align*}
This completes the proof.
\end{proof}

Let any measurable set $A$ in $X$ with $\mu(A) < \infty$.
As is easily observed,
 $$\mu_{\chi_A}(t) = \left\{
 \begin{array}{ll}
  \mu(A) &\text{if } t < 1,\\
  0 &\text{if } t \geq 1,
 \end{array}
 \right.$$
 and hence
 $$\chi_A^\ast(t) = \left\{
 \begin{array}{ll}
  1 &\text{if } t < \mu(A),\\
  0 &\text{if } t \geq \mu(A).
 \end{array}
 \right.$$
Therefore we get that
 $$\chi_A^{\ast\ast}(t) = \frac{1}{t}\int_0^t \chi_A^\ast(s) ds = \left\{
 \begin{array}{ll}
  \frac{1}{t}\int_0^t ds = 1 &\text{if } t < \mu(A),\\
  \frac{1}{t}\int_0^{\mu(A)} ds = \frac{\mu(A)}{t} &\text{if } t \geq \mu(A).
 \end{array}
 \right.$$
By virtue of the Hardy-Littlewood inequality, see \cite[Theorem~3.2.10]{EE}, for each $f \in L^0(X)$,
 $$\int_A f(x) d\mu(x) = \int_X f(x)\chi_A(x) d\mu(x) \leq \int_0^\infty f^\ast(t)\chi_A^\ast(t) dt = \int_0^{\mu(A)} f^\ast(t) dt.$$
We have the following:

\begin{lem}\label{rearrange}
Suppose that a metric measure space $X$ has the Vitali covering property and the $s$-doubling property for every $s \geq r$.
Then there is $c \geq 2$ such that for every integrable function $f \in L^0(X)$ on each measurable subsets whose measure is finite and for every positive number $t > 0$, $(A_rf)^\ast(t) \leq cf^{\ast\ast}(t)$.
\end{lem}

\begin{proof}
Let $c \geq 2$ be a positive number as in Lemma~\ref{distribution}.
Clearly, it holds in the case where $f = 0$ a.e..
In the other case, $f^{\ast\ast} > 0$.
Due to the definition, $f^\ast$ is decreasing,
 and hence
 $$f^\ast(t) = \frac{1}{t}\int_0^t f^\ast(t) ds \geq \frac{1}{t}\int_0^t f^\ast(s) ds = f^{\ast\ast}(t).$$
Observe that
 $$\mu(\{x \in X \mid |f(x)| > f^\ast(t)\}) = \mu_f(f^\ast(t)) \leq t.$$
By Lemma~\ref{distribution},
\begin{align*}
 \mu_{A_rf}(cf^{\ast\ast}(t)) &\leq \frac{1}{f^{\ast\ast}(t)}\int_{\{x \in X \mid |f(x)| > f^{\ast\ast}(t)\}} |f(x)| d\mu(x) \leq \frac{1}{f^{\ast\ast}(t)}\int_{\{x \in X \mid |f(x)| > f^\ast(t)\}} |f(x)| d\mu(x)\\
 &\leq \frac{1}{f^{\ast\ast}(t)}\int_0^{\mu(\{x \in X \mid |f(x)| > f^\ast(t)\})} f^\ast(s) ds \leq \frac{1}{f^{\ast\ast}(t)}\int_0^t f^\ast(s) ds = t,
\end{align*}
 which implies that $(A_rf)^\ast(t) \leq cf^{\ast\ast}(t)$.
\end{proof}

Now it is shown that $A_r$ is bounded on $L^{p,q}(X)$.

\begin{prop}\label{oper.bdd.}
Let $X$ be an $s$-doubling metric measure space with the Vitali covering property for all $s \geq r$.
Suppose that $1 < p \leq \infty$ and $1 \leq q \leq \infty$.
Then there exists $c \geq 2$ such that for any $f \in L^{p,q}(X)$,
 $$\|A_rf\|_E \leq \frac{cp}{p - 1}\|f\|_E.$$
\end{prop}

\begin{proof}
Put $c \geq 2$ be as in Lemma~\ref{distribution}.
Combining Lemma~\ref{rearrange} with Lemma~3.4.6 of \cite{EE}, we can obtain that
\begin{align*}
 \|A_rf\|_{p,q} &= \bigg(\int_0^\infty t^{q/p - 1}(A_rf)^\ast(t)^q dt\bigg)^{1/q} \leq \bigg(\int_0^\infty t^{q/p - 1}(cf^{\ast\ast}(t))^q dt\bigg)^{1/q}\\
 &= c\|f\|_{(p,q)} \leq \frac{cp}{p - 1}\|f\|_{p,q}
 \end{align*}
 if $q < \infty$,
 and that
 $$\|A_rf\|_{p,q} = \sup_{t \in (0,\infty)} t^{1/p}(A_rf)^\ast(t)^q \leq \sup_{t \in (0,\infty)} t^{1/p}cf^{\ast\ast}(t) = c\|f\|_{(p,q)} \leq \frac{cp}{p - 1}\|f\|_{p,q}$$
 if $q = \infty$.
Moreover,
 $$\|A_rf\|_{(p,q)} \leq \frac{p}{p - 1}\|A_rf\|_{p,q} \leq \frac{cp}{p - 1}\|f\|_{(p,q)}.$$
The proof completes.
\end{proof}

We shall show the absolute continuity of $\|\cdot\|_E$.

\begin{prop}\label{abs.conti.}
Let $p \in (1,\infty)$ and $q \in [1,\infty]$.
The norm $\|\cdot\|_E$ is absolutely continuous if $p \in (1,\infty)$ and $q \in [1,\infty)$,
 and it is absolutely continuous with respect to $1$ if $p \in (1,\infty)$ and $q = \infty$.
\end{prop}

\begin{proof}
Fix any sequence $\{A_n\}$ of measurable sets in $X$ such that $\mu(A_n) \to 0$ as $n \to \infty$.
In the case that $p \in (1,\infty)$ and $q \in [1,\infty)$, for any $f \in L^{p,q}(X)$, $(f\chi_{A_n})^\ast(t) \leq f^\ast(t)$ if $t \in [0,\mu(A_n))$,
 and $(f\chi_{A_n})^\ast(t) = 0$ if $t \in [\mu(A_n),\infty)$.
As $n \to \infty$, by the Lebesgue dominated convergence theorem, 
 $$\|f\chi_{A_n}\|_{p,q} = \bigg(\int_0^\infty t^{q/p - 1}(f\chi_{A_n})^\ast(t)^q dt\bigg)^{1/q} = \bigg(\int_0^\infty t^{q/p - 1}(f\chi_{A_n})^\ast(t)^q\chi_{[0,\mu(A_n))}(t) dt\bigg)^{1/q} \to 0.$$
In the case that $p \in (1,\infty)$ and $q = \infty$, as $n \to 0$,
 $$\|\chi_{A_n}\|_{p,q} = \sup_{t \in (0,\infty)} t^{1/p}(\chi_{A_n})^\ast(t) = \sup_{t \in (0,\mu(A_n))} t^{1/p} = \mu(A_n)^{1/p} \to 0.$$
It follows from the equivalence between $\|\cdot\|_{p,q}$ and $\|\cdot\|_{(p,q)}$ that $\|f\chi_{A_n}\|_{(p,q)}$ and $\|\chi_{A_n}\|_{(p,q)}$ also tend to $0$.
We finish the proof.
\end{proof}

For every measurable subset $A \subset X$ of finite measure, when $1 \leq p \leq \infty$ and $1 \leq q < \infty$,
 it follows that
 $$\|\chi_A\|_{p,q} = \bigg(\int_0^\infty t^{q/p - 1}\chi_A^\ast(t)^p dt\bigg)^{1/q} = \bigg(\int_0^{\mu(A)} t^{q/p - 1} dt\bigg)^{1/q} = \bigg(\frac{p}{q}\bigg)^{1/q}(\mu(A))^{1/p},$$
 and that
 \begin{align*}
  \|\chi_A\|_{(p,q)} &= \bigg(\int_0^\infty t^{q/p - 1}\chi_A^{\ast\ast}(t)^p dt\bigg)^{1/q} = \bigg(\int_0^{\mu(A)} t^{q/p - 1} dt + \int_{\mu(A)}^\infty t^{q/p - 1}\bigg(\frac{\mu(A)}{t}\bigg)^q dt\bigg)^{1/q}\\
  &= \bigg(\frac{p^2}{q(p - 1)}\bigg)^{1/q}(\mu(A))^{1/p}.
 \end{align*}
Moreover, when $1 \leq p \leq \infty$ and $q = \infty$,
 observe that
 $$\|\chi_A\|_{p,q} = \sup_{t \in (0,\infty)} t^{1/p}\chi_A^\ast(t) = \sup_{t \in (0,\mu(A))} t^{1/p} = (\mu(A))^{1/p},$$
 and that
 \begin{align*}
  \|\chi_A\|_{(p,q)} &= \sup_{t \in (0,\infty)} t^{1/p}\chi_A^{\ast\ast}(t) = \sup_{t \in [\mu(A),\infty)} t^{1/p - 1}\mu(A) = (\mu(A))^{1/p}.
 \end{align*}
Combining the Hardy-Littlewood inequality with the H\"{o}lder inequality on Lebesgue spaces, we can establish the H\"{o}lder inequality on Lorentz spaces.
Thus for any $p, q \in [1,\infty]$ and for any $f \in L^{p,q}(X)$,
\begin{align*}
 \int_A |f(x)| d\mu(x) &= \int_X |f(x)\chi_A(x)| d\mu(x) \leq \|f\|_{p,q}\|\chi_A\|_{p/(p - 1),q/(q - 1)}\\
 &\leq (\mu(A))^{1 - 1/p}\bigg(\frac{p(q - 1)}{q(p - 1)}\bigg)^{1 - 1/q}\|f\|_{p,q}.
\end{align*}
Now we will prove Corollary~\ref{equiv.lorentz}.

\begin{proof}[Proof of Corollary~\ref{equiv.lorentz}]
By Proposition~\ref{oper.bdd.}, $A_r(L^{p,q}(X)) \subset L^{p,q}(X)$.
Due to the H\"{o}lder inequality and Lemma~3.4.6 of \cite{EE}, for any measurable set $A \subset X$ with $\mu(A) < \infty$,
 $$\int_A |f(x)| d\mu(x) \leq (\mu(A))^{1 - 1/p}\bigg(\frac{p(q - 1)}{q(p - 1)}\bigg)^{1 - 1/q}\|f\|_E.$$
Letting $\lambda = (p(q - 1)/(q(p - 1)))^{1 - 1/q}$, we consider $\alpha(A) = \lambda(\mu(A))^{1 - 1/p}$ in the condition (B5).
First, we will prove the ``if'' part.
Using Lemma~\ref{doubling}, we have that $\mu(X) < \infty$ and $\inf\{\mu(B(x,r)) \mid x \in X\} > 0$ according to the boundedness and the $s$-doubling property of $X$ for every $s \geq r$.
By the assumption, for each $x \in X$,
 $$\alpha(B(x,r) \triangle B(y,r)) = \lambda(\mu(B(x,r) \triangle B(y,r)))^{1 - 1/p} \to 0$$
 as $y \to x$,
 and moreover $\mu(B(y,r)) \to \mu(B(x,r))$, refer to \cite[Corollary~2.1]{GG}.
Recall that the norm $\|\cdot\|_E$ is absolutely continuous by Proposition~\ref{abs.conti.}.
Therefore the ``if'' part follows from Theorem~\ref{suff.}.

Next, we shall show the ``only if'' part.
Due to the $r$-doubling property of $X$, there is $\gamma \geq 1$ such that $\mu(B(x,2r)) \leq \gamma\mu(B(x,r))$ for any $x \in X$.
Observe that
\begin{align*}
 \bigg\|\frac{\alpha(B(x,r))\chi_{B(x,2r)}}{\mu(B(x,r))}\bigg\|_E &= \bigg\|\frac{\lambda(\mu(B(x,r)))^{1 - 1/p}\chi_{B(x,2r)}}{\mu(B(x,r))}\bigg\|_E = \lambda(\mu(B(x,r)))^{-1/p}\|\chi_{B(x,2r)}\|_E\\
 &= \left\{
 \begin{array}{l}
  \lambda(\mu(B(x,r)))^{-1/p} \cdot \Big(\frac{p}{q}\Big)^{1/q}(\mu(B(x,2r)))^{1/p} \leq \frac{p}{q}\Big(\frac{q - 1}{p - 1}\Big)^{1 - 1/q}\gamma^{1/p}\\
  \ \ \ \ \ \ \ \ \ \ \ \ \ \ \ \ \ \ \ \ \ \ \ \ \ \ \ \ \ \ \ \ \ \ \ \ \ \ \ \ \ \ \ \ \ \ \ \ \ \ \ \ \ \ \ \ \text{if } \|\cdot\|_E = \|\cdot\|_{p,q},\\
  \lambda(\mu(B(x,r)))^{-1/p} \cdot \Big(\frac{p^2}{q(p - 1)}\Big)^{1/q}(\mu(B(x,2r)))^{1/p} \leq \frac{p(q - 1)}{q(p - 1)}\Big(\frac{p}{q - 1}\Big)^{1/q}\gamma^{1/p}\\
  \ \ \ \ \ \ \ \ \ \ \ \ \ \ \ \ \ \ \ \ \ \ \ \ \ \ \ \ \ \ \ \ \ \ \ \ \ \ \ \ \ \ \ \ \ \ \ \ \ \ \ \ \ \ \ \ \text{if } \|\cdot\|_E = \|\cdot\|_{(p,q)}.
 \end{array}
 \right.
\end{align*}
Applying Theorem~\ref{nec.}, we can prove the ``only if'' part.
The proof is completed.
\end{proof}

\subsection*{Acknowledgements}

The author would like to thank Professor Przemys{\l}aw G\'{o}rka,
 who suggested this problem, for his useful comments.

\end{document}